\documentclass{amsart}
\usepackage{graphicx}
\usepackage{amssymb,amscd,amsthm,amsxtra}
\usepackage{latexsym}
\usepackage{epsfig}
\newtheorem{thm}{Theorem}[section]
\newtheorem{cor}[thm]{Corollary}
\newtheorem{lem}[thm]{Lemma}
\newtheorem{prop}[thm]{Proposition}
\theoremstyle{definition}
\newtheorem{defn}[thm]{Definition}
\theoremstyle{remark}
\newtheorem{rem}[thm]{Remark}
\numberwithin{equation}{section}
\newcommand{\R}{\mathbb R}

\newcommand{\eps}{\epsilon}

\newcommand{\p}{\partial}

\newcommand{\comment}[1]{}
\begin{document}
\title{Quasi-Harnack inequality}

\author{D. De Silva}
\address{Department of Mathematics, Barnard College, Columbia University, New York, NY 10027, USA}
\email{\tt  desilva@math.columbia.edu}
\author{O. Savin}
\address{Department of Mathematics, Columbia University, New York, NY 10027, USA}\email{\tt  savin@math.columbia.edu}
\thanks{O.~S.~is supported by  NSF grant DMS-1200701.}

\begin{abstract}
In this paper we obtain some extensions of the classical Krylov-Safonov Harnack inequality. The novelty is that we consider functions that do not necessarily satisfy an infinitesimal equation but rather exhibit a two-scale behavior. We require that at scale larger than some $r_0>0$ (small) the functions satisfy the comparison principle with a standard family of quadratic polynomials, while at scale $r_0$ they satisfy a Weak Harnack type estimate. 

We also give several applications of the main result in very different settings such as discrete difference equations, nonlocal equations, homogenization and the quasi-minimal surfaces of Almgren.  

\end{abstract}

\maketitle
\section{Introduction}

 The Harnack inequality of Krylov and Safonov plays a fundamental role in the regularity theory of nonlinear second order elliptic equations. It states that a nonnegative solution $u \ge 0$ to a second order equation satisfies
$$a^{ij}(x)u_{ij}=0 \quad \mbox{in $B_1$}  \quad \Longrightarrow \quad \quad \sup_{B_{1/2} }u \le C u(0),$$
 where $C$ is a {\it universal} constant depending only on $n$ and the ellipticity constants $\lambda_{min}$, $\lambda_{max}$ of the coefficients $a^{ij}(x)$. It is a quantified version of the maximum principle and it provides the $C^\alpha$ estimates and the compactness for the class of solutions to uniformly elliptic equations with measurable coefficients.

Harnack inequality has a different version for supersolutions (see Lemma 4.5 in \cite{CC}). It is a pointwise-to-measure estimate known as Weak Harnack inequality and it can be stated in the following way.

{\bf Weak Harnack inequality.} {\it Assume that $u$ is a positive supersolution
$$ u \ge 0, \quad a^{ij}(x)u_{ij}\le0 \quad \quad \mbox{in $B_1$, and}\quad \quad u(0)=1.$$
Given $\delta>0$ small, there exists $C(\delta)$ large depending on $\delta$ and the universal constants such that
$$|\{u \leq C(\delta) \}| \ge (1-\delta) |B_1|.$$}

An interesting observation is that the class of supersolutions in the Weak Harnack inequality can be enlarged to include functions that might not necessarily satisfy any equation. The proof in \cite{CC} carries through if we require $u$ to satisfy the comparison principle by below only with one quadratic polynomial
$$P_\Lambda(x):= \frac{\Lambda}{2} x_n^2 - \frac 12 |x'|^2,$$
together with rotations and dilations of it, with $\Lambda$ some large fixed constant. 

In this paper we investigate the situation when the comparison with the family of functions $P_\Lambda$ is satisfied only up to a scale $r_0$ (small), and at scale $C(\Lambda)r_0$, a version of weak Harnack inequality holds. In our main result we show that the Harnack estimates still remain valid with constants independent of $r_0$. 

The second part of the paper is devoted to show how our results can be applied in several different situations and we discuss four representative examples. 

First, as a direct application, we recover the Harnack inequality of Hung-Ju and Trudinger \cite{HT} for discrete difference uniformly elliptic equations.   

A second motivating example comes from the theory of nonlocal equations. Solutions to integro-differential equations of order $\sigma \in (0,2)$, when $\sigma$ is close to 2, belong to a family with a two-scale behavior. At small scales the integral behavior dominates whereas at large scales the equation becomes close to being a local equation. We deduce the nonlocal version of the Harnack inequality of Caffarelli and Silvestre \cite{CS} (with uniform constants as $\sigma \to 2^-$) from our results. There are many other problems involving nonlocal operators which fit the general framework given in this paper. We mention for example the uniform H\"older estimates of Caffarelli and Valdinoci for nonlocal $s$-minimal surfaces as $s \to \frac 12^-$, or the boundary layer estimates for phase transitions from \cite{S}.

A third application of our results comes from the homogenization of degenerate PDEs. 
A simple situation we consider is the case of linear equations with measurable coefficients with ellipticity constants $\lambda_{min}(\frac x \eps)$, $\lambda_{max}( \frac x \eps)$ for some functions $\lambda_{min}$, $\lambda_{max}$ periodic of period 1. Notice that if $\lambda_{min}$, $\lambda_{max}$ degenerate in the interior of the unit cube but they are bounded away from $0$ and $\infty$ near the sides of the cube, then Harnack inequality holds at scale $\eps$. It turns out that we end up in a setting as above and the uniform Holder estimates remain valid at all scales (see Theorem \ref{THo}).

Finally, we discuss the case of quasi-minimizers for the perimeter functional for sets which are sufficiently close to half-spaces, and we show that quasi-minimizers exhibit a two-scale type behavior. Then we sketch a nonvariational proof of the $C^{1,\alpha}$ estimates of Almgren and Tamanini based on Harnack inequality.

    The paper is organized as follows. In Section 2 we introduce the classes of functions $\mathcal P^I_\Lambda(r)$ and $ \mathcal W^a_M(\rho)$ and we state the main results of the paper Theorem \ref{main} and Proposition \ref{HI}. Section 3 is devoted to the proof of the main results. In Section 4 we discuss some of the applications mentioned above, and finally in Section 5 we consider the case of quasi-minimizers. 
 
\section{Statement of the main results}
Let $\Lambda>0$ be fixed (large) and let
$$P_\Lambda(x):= \frac \Lambda 2 (x \cdot \xi)^2 - \frac 12  |x|^2 + b \cdot x + d, \quad \quad \xi, b \in \R^n, \quad d \in \R,$$ 
be a quadratic polynomial with $\xi$ a unit direction 
$$|\xi|=1, \quad \mbox{and} \quad |b|, |d| \le 1.$$
Let $r >0,$ and $I \subset [0, \infty)$ be a closed interval. 

We introduce the following definition for $\mathcal P _\Lambda^I(r)$ which can be thought as {\it the class of supersolutions} of size $I$ at scale $r$.

\begin{defn}\label{super_comp} Let $u: \Omega \to \R$ be a continuous function on a domain $\Omega \subset \R^n$. We say that
$$u \in {\mathcal P}_{\Lambda}^{I}(r) \quad \quad \mbox{in $\Omega$},$$ if and only if $u$ cannot be touched from below at any point $x_0$ in a neighborhood $B_{r}(x_0) \subset \Omega$ by a polynomial $a \, P_\Lambda(x-x_0)$ as above, for some $a \in I$.
\end{defn}

We recall that if $u$ and $\varphi$ are continuous functions in a domain $\Omega$, we say that $\varphi$ touches $u$ from below (resp. above) at $x_0 \in \Omega$ in a neighborhood $B_r(x_0)$ whenever
$$\varphi \leq u \quad (\text{resp. $\varphi \geq u$}) \quad \text{in $B_r(x_0)$}, \quad \varphi(x_0)=u(x_0).$$

Notice that 
$$ {\mathcal P}_{\Lambda_1}^{I}(r) \subset {\mathcal P}_{\Lambda_2}^{I}(r) \quad \mbox{if} \quad \Lambda_1 \le \Lambda_2,$$
and in our analysis we assume without loss of generality that $\Lambda$ is large.

Now, for parameters $M, a, \rho>0$ we introduce another class of functions ${\mathcal{W}}_{M}^a(\rho)$ which can be thought as the class of functions of size $a$ which satisfy Weak Harnack inequality at scale $\rho$.

\begin{defn}\label{whi} Let $u: \Omega \to \R$ be a continuous function on a domain $\Omega \subset \R^n$. We say that 
$$u \in {\mathcal{W}}_{M}^a(\rho) \quad \quad \mbox{in $\Omega$}, $$ if $u \ge 0$ and whenever $B_{2\rho}(x_0) \subset \Omega$,
and 
$$u(x_0) \leq a, $$
then 
$$\frac{|\{u \leq M\, a\} \cap B_{\rho}(x_0)|}{|B_{\rho}(x_0)|} \geq 1-\delta,$$ with $\delta(n)>0$ a given constant depending only on $n$. 
\end{defn}

The value of $\delta(n)$ in the definition above will be specified later, in the proof of Theorem \ref{HI}.

Let $\Lambda, M>0$  be given. Our main theorem reads as follows. 

\begin{thm}\label{main} Let $u$ be a continuous function in $B_1$, with $0 \leq u \leq 1.$ There exist constants $c_0, \bar C,$ depending only on $n$ and $\Lambda$ and positive constants $a(M)$,$\eta(M)$ depending also on $M$ such that if 
$$1-u,u \; \in \; \mathcal P_{\Lambda}^{[\eta, 1]}(r) \cap \mathcal W_{M}^{a}(\bar C r)$$
for some $r \leq c_0$, then  
$$osc_{B_{1/2}} u \leq 1-\eta.$$
\end{thm}

The proof of Theorem \ref{main} is based on the following version of Weak-Harnack Inequality.

\begin{thm}\label{HI}  Let $u$ be a continuous function in $B_1$, with $u \geq 0$ and assume that $$u(\bar x) \leq 1 \quad \text{for some $\bar x \in B_{1/2}$.}$$ There exist constants $ c_0, C_0,\bar C >0$ depending only on $n$ and $\Lambda$ such that 
if 
$$u \;  \in \; {\mathcal P}_{\Lambda}^{[1,C_0]}(r) \cap {\mathcal W}_{M}^{C_0}(\bar C r)$$
for some $r \leq c_0$, 
then  
$$|\{u \leq 2 C_0 M \} \cap B_{1/2}| \geq \frac 3 4 |B_{1/2}|.$$

\end{thm}

Once Theorem \ref{HI} is established, the proof of Theorem \ref{main} follows from standard arguments 
by choosing $a=\frac{1}{4M}$ and $\eta= \frac{1}{4C_0M}$.

Notice that the estimates of Theorems \ref{main} and \ref{HI} do not depend on $r$. As $r \to 0$ we recover the standard Harnack inequality for viscosity solutions since any continuous function $u$ satisfies for a fixed $a>0$, $u \in \mathcal W^a_2(r)$ (with $\delta=0$) for all $r$ sufficiently small (depending on the modulus of continuity of $u$).

We also obtain a version of Theorem \ref{HI} which applies to multivalued graphs or, more generally, to closed sets $\Gamma \subset \R^{n+1}$ (see Proposition \ref{PP} and Remark \ref{r3}).

In the second part of the paper we provide a variety of applications of our results. 
We may use slightly weaker definitions for the classes of function $\mathcal P$ and $\mathcal W$ which we make precise below. This will be useful whenever we deal with nonlocal problems for which the global behavior of $u$ in $B_1$ is also relevant.

To this aim, first let us denote by
$$P^a_y(x):= P^a(x-y) +c_y, \quad P^a(x):= -\frac{a}{2}|x|^2, \quad \quad a>0, \quad c_y>0,$$
the class of paraboloids of opening $-a$ and vertex $y$ which have the additional property that
$$P^a_y(x) \le 0 \quad \mbox{outside $B_{3/4}$}.$$
It is straightforward to check that this property implies that 
\begin{equation} \label{01}
y \in B_{3/4}, \quad P^a_y \le a, \quad |\nabla P^a_y| \le a \quad \mbox{in the set where $P_y^a\ge 0$.} 
\end{equation}

\begin{defn}\label{p2} We say that a continuous function in $B_1$, $u \ge 0$,  belongs to the class 
$\mathcal {P}^{I}_\Lambda(r)$ whenever $u$ cannot be touched from below at any point $x_0$ by
$$P_y^a(x) + a \frac \Lambda 2 ((x-x_0)\cdot \xi)^2 \chi_{B_r(x_0)} \quad \mbox{in} \quad B_1,$$
with $a \in I$. Here $P_y^a$ is a quadratic polynomial as above and $\xi$ is a unit direction.
\end{defn}

The set of all contact points (from below) between a given continuous function $u \ge 0$ and a paraboloid $P_y^a$ of the class above will be denoted by $A_a(u).$ Precisely,
\begin{equation}\label{A}
A_a(u):=\{z \in B_1: \exists \ P^a_y \ \text{touching $u$ from below at $z$ in $B_1$}\},
\end{equation}
and we have that $A_a(u)$ is a compact set included in $\overline B_{3/4}$, $u \le a$ on $A_a(u)$ and
$$A_{a_1}(u) \subset A_{a_2}(u) \quad \mbox{if $a_1 \le a_2$.}  $$

\begin{figure}[h]
\includegraphics[width=0.7 \textwidth]{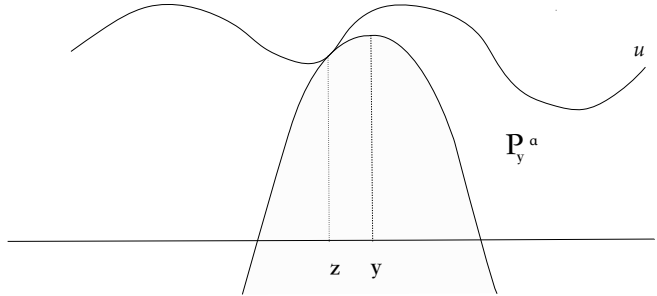}
\caption{$z \in A_a(u)$}
    \label{fig1}
\end{figure}

The difference between Definition \ref{p2} and the original Definition \ref{super_comp} is that above we require in addition that the comparison with $a \, P_\Lambda (x-x_0)$ should hold only at points $x_0$ in $A_a(u)$.  
It is clear from \eqref{01} that if a nonnegative function satisfies Definition \ref{super_comp} in $B_1$ then it satisfies also Definition \ref{p2}.

We modify similarly Definition \ref{whi} to hold only at contact points in $A_a(u)$. 

\begin{defn}\label{w2}
We say that $u \in {\mathcal W}_{M}^a(\rho)$ in $B_1$ if $u \ge 0$ and whenever
 $ x_0 \in A_a(u)$ with touching polynomial $P_y^a$, then
$$\frac{|\{u \leq P_y^a + Ma\}\cap B_{\rho}(x_0)|}{|B_{\rho}(x_0)|} \geq 1- \delta(n).$$
with $\delta(n)$ a small given constant depending only on $n$.
\end{defn}
Clearly, $$ {\mathcal W}_{M}^{a_1}(\rho) \subset {\mathcal W}_{M}^{a_2}(\rho) \quad \mbox{if $a_1 \le a_2$.} $$

\section{Proof of Theorems \ref{main} and \ref{HI}.}

We present here the proof of our main results.
We follow the lines of the standard proof from \cite{CC} and we use an Alexandrov-Bakelman-Pucci (ABP) type estimate combined with the Calderon-Zygmund cube decomposition. The difference in our case is that the estimates hold only at scale larger than $r$ and we need to use some discrete versions of the arguments. We provide the details below.

Throughout this section, {\it universal constants} mean positive constants that depend only on $n$ and $\Lambda$.

Our starting point is the study of the contact set $A_a(u)$ defined in \eqref{A} for a function $u$ in the class ${\mathcal P}_{\Lambda}^a(r)$.

\begin{lem}\label{touching} Let $z_1,z_2 \in A_a(u) \cap B_{5/8}$ and let $y_1,y_2$ be the corresponding vertices of the touching paraboloids.
If $u \in  {\mathcal P}_{\Lambda}^{2a}(r)$ for some $r>0$, then
$$|z_1-z_2| \geq c_0 |y_1-y_2|$$ with  $c_0 >0$ universal, as long as \begin{equation}\label{ver}|y_1 - y_2| \geq C_0 r\end{equation} for a $C_0 >0$ universal.
\end{lem}
\begin{proof} 
After dividing by $a$ we may assume that $a=1$.
Denote $$|y_1-y_2|=:d,$$ and assume by contradiction that the corresponding touching points $z_1, z_2$ 
are at at distance less than $c_0 d$  with $c_0$ to be made precise later. 
Then, we show below that $u$ can be touched by below 
at a point $x_0 \in A_2(u)$ in a $B_{r}(x_0)$ neighborhood 
by a polynomial of the form  $2 P_\Lambda(x-x_0)$, 
and we contradict the fact that $u \in {\mathcal P}_{\Lambda}^2(r)$.

Indeed, let $\Pi$ be the hyperplane determined by the intersection $\{P^1_{y_1}=P^1_{y_2}\}$ 
of the touching paraboloids at $z_1,z_2$ respectively. 
Let us assume, without loss of generality that the line segment joining $z_1$ and $z_2$ 
intersects $\Pi$ at $0$ and that $y_2-y_1=d \, e_n.$
In particular, $$\Pi: =\{x_n=0\}, \quad z_1 \in \{x_n \leq 0\}, \quad z_2 \in \{x_n \geq 0\}$$
and
$$|z_i| \leq c_0 \, d, \quad i=1,2.$$

Denote by $\tilde P$ the quadratic polynomial 
$$\tilde P(x):= P^1_{y_1}(x) + \frac{d}{2} x_n + \Lambda x_n^2 - \frac 1 2 |x|^2 +  c_0 d^2,$$
and notice that $\tilde P$ can be written as
$$\tilde P= P^2_{(y_1+y_2)/4} + \Lambda x_n^2 + const.$$

Let $$v(x):= u(x)-\tilde P(x),$$
and let $D$ denote the annular region
$$D:=\{c_1 d \leq |x| \leq  4c_1 d \},$$ with $c_1:=c_0^\frac 13 \gg c_0$ small, universal. We claim that 
\begin{equation}\label{vsmall}v \ge \max_{i=1,2} P^1_{y_i} - \tilde P > 0 \quad \text{on $D$, and} \quad v(z_1) <0.\end{equation}
This shows that $\tilde P + \min_{B_{4c_1d}}v$ touches $u$ from below 
at a point $x_0 \in B_{c_1 d}$ in the neighborhood $B_{4c_1d} \supset B_{3c_1d}(x_0)$. 
We contradict the hypothesis $u \in \mathcal P^2_\Lambda(r)$ 
provided that we show $x_0 \in A_2(u)$ and that 
we choose $C_0 = (3c_1)^{-1}$ in \eqref{ver} so that $3 c_1 d \geq r$.

Next we check that $x_0 \in A_2(u)$. 
Let $P^2_{\bar y}$ be the polynomial of opening $-2$ tangent to $\tilde P + \min_{B_{4c_1d}}v$ at $x_0$. 
Then $P^2_{\bar y} \le u$ in $B_{4c_1 d}$ and $P^2_{\bar y} \le \max_{i=1,2} P^1_{y_i}$ outside $B_{4c_1 d}$. 
The last assertion follows from the fact that 
the difference between $\max P^1_{y_i}$ and $P^2_{\bar y}$ is convex 
and by \eqref{vsmall} it is positive on $D$ and negative at $x_0$, therefore it must be positive also outside $B_{4c_1 d}$. 
In conclusion $P^2_{\bar y}$ touches $u$ by below at $x_0$ in the whole domain $B_1$ 
and also $P^2_{\bar y} \le 0$ outside $B_{3/4}$, hence $x_0 \in A_2(u)$.

It remains to prove the claim \eqref{vsmall}. For the first inequality we restrict, say, to the region $D \cap \{x_n \leq 0\}$, and we have that either
$$x_n \leq -\frac{c_1 d}{2} \quad \mbox{or} \quad  |x'| \geq \frac{c_1 d}{2}.$$
In the first case,
$$P_{y_1}^1 - \tilde P \ge   \left(\frac{c_1}{4}  - 16 \Lambda c_1^2 -  c_0\right)  d^2 >0,$$ 
by choosing $c_1$ sufficiently small. In the second case
$$P_{y_1}^1 - \tilde P \ge -x_n \left(\frac d2 + \Lambda x_n \right) + \left(\frac{c_1^2}{8} -  c_0 \right)d^2 > 0.$$
For the second part of \eqref{vsmall} we compute
$$v(z_1)=  P^1_{y_1}(z_1)- \tilde P(z_1)  \leq \left( \frac{c_0}{2}  + \frac 1 2 c_0^2  - c_0 \right) d^2 \leq  0.$$

\end{proof}

In what follows, 
we work on cubes $Q_\rho(\tilde x)$ of size $\rho$ defined as:
$$Q_\rho(\tilde x) := \{x \in \R^n : |x_i - \tilde x_i| < \rho/2, \quad i=1,\ldots, n.\}$$ Sometimes, for simplicity of notation, we drop the dependence on the center $\tilde x$ when there is no possibility of confusion.

We decompose the space in dyadic cubes of size $2^{-l}$, $l=0,1,2,...$ 
and by abuse of notation, denote by $Q_l$ a  cube of size $2^{-l}$ obtained from this decomposition. Set,
$$A_a^l(u):= \bigcup_{\overline{Q_l} \cap A_a(u) \neq \emptyset} Q_l,$$
that is $A_a^l(u)$ denotes the collection of the dyadic cubes of size $2^{-l}$ which intersect $A_a(u)$.

 Then, with this notation
Lemma \ref{touching} gives the following corollary.

\begin{cor} \label{percent}Assume $Q_\rho=Q_\rho(\tilde x) \subset B_{5/8}$ and that $$A_a(u) \cap Q_{\rho/2}(\tilde x) \neq \emptyset.$$ 
If $u \in {\mathcal P}_{\Lambda}^{4a}(r)$ and $r\le 2^{-l}$, then
$$|A_{2a}^l(u) \cap Q_\rho| \geq \eta |Q_\rho|,$$ for a universal constant $\eta.$
\end{cor}
\begin{proof} Let  $\bar x \in A_a(u) \cap Q_{\rho/2}$ and let $P^a_{\bar y}$ be a corresponding touching paraboloid at $\bar x$. For a set of vertices $y \in B_{\rho/10}(\bar x)$, we consider the polynomial 
$$\tilde P(x):= P_{\bar y}^a(x) - \frac a 2 |x-y|^2 + c_y,$$
 with $c_y$ chosen so that $\tilde P$ becomes tangent to $u$ from below. Notice that
$ \tilde P(x)=P_{(y+\bar y)/2}^{2a}(x)$.
Since $y \in B_{\rho/10}(\bar x)$ we easily conclude that $\tilde P \le P^a_{\bar y}$ outside $B_{\rho/2}(\bar x)$. 
This means that the corresponding contact point between the polynomial $\tilde P$ and $u$ 
lies in $Q_\rho \cap A_{2a}(u)$.
 
 We now select vertices $y$ in $B_{\rho/10}(\bar x)$, with $\bar x$ among them, at distance larger than $2 C_0 2^{-l}$ from each other. By Lemma \ref{touching}, the corresponding contact points will also be at distance larger than $c_0 2^{-l}$ from each other, with $c_0, C_0$ as above. Since we can select  $C \rho^n 2^{ln}$ such vertices, we have at least as many contact points at distance larger than $c_0 2^{-l}$. Then our statement follows immediately from the definition of $A^l_{2a}(u).$
\end{proof}
Next, we obtain a generalized version of the corollary above, in which $A_a(u)$ intersects 
$Q_{3 \rho}$  instead of $Q_{\rho/2}$.

\begin{lem}\label{general} Assume $B_{6 \sqrt n \rho }(\tilde x) \subset B_{5/8}$ and that 
$$A_{a}(u) \cap Q_{3 \rho}(\tilde x) \neq \emptyset.$$
If $u \in  {\mathcal P}_{\Lambda}^{[a, C_1 a]}(r)$ with $r \leq \min\{c_1 \rho, 2^{-l}\},$ then
$$|A^l_{C_1 a}(u) \cap Q_\rho| \geq \eta |Q_\rho|,$$
with $C_1$, $c_1$ and $\eta$ universal.
\end{lem}
\begin{proof} After a dilation, translation and multiplication with a constant we may assume that $\tilde x=0$, $\rho=1$ and $a=1$.
In view of Corollary \ref{percent}, it suffices to show that there exists a point 
$x_0 \in A_{C}(u) \cap Q_{1/2}$, for $C$ large universal.

To prove the existence of $x_0$, we define the following barrier:
\begin{equation}\label{barrier}
\phi(x):= |x|^{-\gamma}-\beta, \quad \quad x \neq 0
\end{equation}
with $\gamma=4\Lambda$,
$\beta>0$ chosen so that
$$\phi(x)=0 \quad \text{on $\p B_{6\sqrt{n}}.$}$$
Denote by $P_\sigma$ the quadratic polynomial tangent to $\phi$ by below on the sphere of radius $\sigma>0$.
Using the homogeneity of $|x|^{-\gamma}$, we find  
\begin{equation}\label{sd}
\phi \ge P_\sigma=P^d_0 \quad \mbox{ with $d=\gamma \sigma^{-\gamma-2}$}
\end{equation}
and if $|x_0|=\sigma$, then by the Taylor expansion of $\phi$ near $x_0$ we get
\begin{equation}\label{Taylor}(\phi - P_{x_0/2}^{2d})(x) \ge \frac{(\gamma+1) d}{2}  ((x-x_0)\cdot \xi )^2,   \quad \mbox{if} \quad  x \in B_{\bar c \sigma}(x_0),\quad \quad \xi=\frac {x_0}{|x_0|},\end{equation}
with $\bar c, $ sufficiently small universal.

Let 
$$\tilde \phi(x)= \phi(x) \quad \text{in $|x| \geq \frac 1 4$}, \quad \tilde \phi(x) = P_{1/4}, \quad \text{on $B_{1/4}$.}$$ 
For $\bar x \in A_1(u) \cap Q_3,$ we call $P^1_{\bar y}$ a corresponding touching polynomial of vertex $\bar y.$

We slide $\phi_t= P^1_{\bar y} + \alpha \tilde \phi + t$, (for some constant $\alpha>0$) from below till it touches $u$ for the first time at some point $x_0$. Since $\phi \leq 0$ outside $B_{6\sqrt n}$ and $u\geq P^1_{\bar y}$ with equality at $\bar x \in Q_3 \subset B_{6\sqrt n}$, the first contact point $x_0$ occurs for some value of $t \le 0$ and $x_0$ must belong to $B_{6\sqrt n}$.

We claim that $x_0 \in B_{1/2}.$ Otherwise, at $x_0$ we have a tangent polynomial of opening $a_0:=1+ 2 \alpha d$ by below
with $\sigma=|x_0|$ and $d$ as in \eqref{sd}, and 
\begin{equation}\label{alpha}\alpha (\gamma +1)d \ge \Lambda (1+2 \alpha d)=\Lambda a_0,\end{equation}
provided that $\alpha$ is chosen sufficiently large universal. 
We use \eqref{Taylor}-\eqref{alpha} together with the fact that $|x_0| \geq \frac 12$ to contradict that $u \in {\mathcal P}_{\Lambda}^{[1,C_1]}(r)$,  as long as  $r \le c_1 \le \bar c / 2$, and $C_1 \geq 1+ 2 \alpha \gamma 2^{\gamma+2} \ge a_0.$

Thus the touching occurs in the  ball $B_{1/2}\subset Q_1$  and the desired point $x_0$ is obtained.

\end{proof}

A standard iteration argument and the Calderon Zygmund cube decomposition give the following result.

\begin{lem} \label{iteration} 

Assume $B_{6 \sqrt n \rho }(\tilde x) \subset B_{5/8}$ and $Q_\rho(x)$ is a dyadic cube such that 
$$A_{a}(u) \cap Q_{3 \rho}(\tilde x) \neq \emptyset.$$
If $u \in  {\mathcal P}_{\Lambda}^{[a, C^k_1 a]}(r)$ with $r \leq c_1 2^{-l}, and $ $2^{-l} \le \rho$, then
$$|A^l_{C^k_1 a}(u) \cap Q_\rho| \geq (1-(1-\eta)^k) |Q_\rho|.$$

\end{lem}
\begin{proof} After a dilation and a multiplication by a constant we may assume as above that $\rho=1$, $\tilde x=0$, $a=1$. The case $k=1$ holds by Lemma \ref{general}.

For $k=2$ we apply the Calderon Zygmund decomposition to the set $A_{C_1}(u)$ in $Q_1$: we decompose dyadically the cube $Q_1$, up to scale $2^{-l}$, and at each step $j$ in our decomposition, we only split the cubes $\tilde Q_j$ which have non-empty intersection with the set of contact points $A_{C_1}(u)$. A remaining cube in the decomposition (which has empty intersection with $A_{C_1}(u)$) will be called ``clean". Let $\tilde Q_j$ be a clean cube of size $2^{-j} \geq 2^{-l},$ and by construction the dilation of factor $3$ around the center $\tilde Q_j$ intersects $A_{C_1}(u)$. Lemma \ref{general} applied to the cube $\tilde Q_j$ with $a=C_1$ gives
$$|A_{C_1^2}^l(u) \cap \tilde Q_j| \geq \eta |\tilde Q_j|.$$
Since this inequality holds for all clean cubes, we conclude that
$$|A^l_{C_1^2}(u) \cap (Q_1 \setminus A_{C_1}^l(u))| \geq \eta (|Q_1| - |A_{C_1}^l(u)|).,$$
which gives the desired conclusion. The general case follows by the argument above $k-1$ times.
\end{proof}

Next we use a covering argument and obtain the following version of Lemma \ref{iteration} for functions in the general class $\mathcal P_\Lambda^I(r)$.

\begin{prop}\label{PP}
 Let $u$ be a continuous function in $B_1$, with $u \geq 0$ and assume that $$u(\bar x) \leq 1 \quad \text{for some $\bar x \in B_{1/2}$.}$$ Given $\mu>0$, there exist constants $\bar c, \bar C >0$ depending only on $n$ and $\Lambda$ and $C(\mu)$ depending also on $\mu$ such that if
$$u \; \in \; {\mathcal P}_{\Lambda}^{I}(r), \quad \quad \mbox{with} \quad I=[1,C(\mu)],$$
for some $r \leq \bar c$, 
then  
$$|A^{l}_{C(\mu)}(u) \cap B_{1/2}| \geq (1-\mu) |B_{1/2}|, \quad \quad \mbox{with} \quad  \bar C r \le 2^{-l}.$$
\end{prop}

The proposition states that if we decompose the space in cubes of size $\bar C r$ and then we consider the collection of those cubes which intersect the contact set 
$A_{C(\mu)}(u)$ then they cover all but a $\mu$-fraction of $B_{1/2}$. We recall that $A_{C(\mu)}(u)$ represent the set of points where $u$ admits a tangent polynomial of opening $- C(\mu)$ by below and which also are less than $0$ outside $B_{3/4}$.

Clearly Proposition \ref{PP} follows from Lemma \ref{iteration} since the existence of $\bar x$ implies
$A_{C}(u) \cap B_{1/16}(\bar x) \ne \emptyset$ for some large $C$, and then we can cover $B_{1/2}$ with a finite collection of dyadic cubes $Q_\rho$, with $\rho$ universal, for which we can apply Lemma \ref{iteration}.

\begin{rem}\label{r3}
It is straightforward to generalize Definition \ref{p2} to include also closed sets $\Gamma \subset B_1 \times \R_+ \subset \R^{n+1}$ (which above can be thought as of the graph of $u$), and to define the corresponding contact sets $A_a(\Gamma)$ and $A^l_a(\Gamma)$.
Then Proposition \ref{PP} holds also in the setting of sets $\Gamma \in \mathcal P^I_\Lambda(r)$ since we did not use so far the graph property of $u$.  
\end{rem}
We are now ready for the proof of Proposition \ref{HI}.

\

{\it Proof of Proposition $\ref{HI}$.} We choose $\mu = 1/8$ in Proposition \ref{PP} and let $C_2=C(\mu)$ and $l$ to be the largest $l$ for which $\bar C r \le 2^{-l}$. Then
\begin{equation}\label{31}
|A_{C_2}^l(u)\cap B_{1/2}|\geq \frac{7}{8} |B_{1/2}|.
\end{equation}
Now, let $\tilde x \in Q_l \cap A_{C_2}(u)$ and assume $u \in \mathcal W_{M}^{C_2}(2 \sqrt n \bar C r)$ (see Definition \ref{w2}). 
Then
$$|\{u \leq 2 M C_2\} \cap B_{2 \sqrt n \bar C r}(\tilde x)| \geq (1-\delta)|B_{2 \sqrt n \bar C r}(\tilde x)|,$$
which implies
$$|\{u \leq 2 M C_2\} \cap Q_l| \ge \frac {9}{10}|Q_l|,$$
provided that $\delta=\delta(n)$ is chosen sufficiently small. Summing over all cubes $Q_l$ included in $A^l_{C_2}(u)$, together with \eqref{31} and the fact that $l$ is sufficiently large gives
$$ |\{u \leq 2 M C_2\} \cap B_{1/2}| \ge \frac {3}{4}|B_{1/2}|.$$
The proposition is proved after relabeling the constants.
\qed

\section{Applications}
\subsection{Discrete equations.} Harnack inequality for discrete difference equations was obtained by Hung-Ju and Trudinger in \cite{HT}. Here we verify that the setting we developed for the classes $\mathcal P$ and $\mathcal W$ applies in this context. We consider
for simplicity the discrete operator involving only the neighbors along the coordinate axes ($\eps>0$ small)
\begin{equation}\label{discrete}
Lu(x):= \sum_{i=1}^n \lambda_i(x) u^\eps_{ii}(x), \quad x\in B_1 \cap \eps \mathbb{Z}^n,
\end{equation}
$$ u^\eps_{ii}(x):= \frac{u(x+\eps e_i)+ u(x-\eps e_i) -2 u(x)}{\eps^2},$$
with 
\begin{equation}\label{discrete_elliptic}
0<\lambda_{min} \leq \lambda_i(x) \leq \lambda_{max}, \quad x\in B_1 \cap \eps \mathbb{Z}^n.
\end{equation}
Below we show that Theorem \ref{main} gives the H\"older continuity for bounded solutions of $Lu=0.$ 

\begin{thm} Assume that 
$$Lu=0, \quad |u|\leq 1 \quad \text{in $B_1$}.$$ Then,
$$|u(x)-u(y)| \leq C |x-y|^\gamma, \quad \text{in $B_{1/2}\cap \eps \mathbb Z^n$},$$
with $C, \gamma$ universal independent of $\eps.$
\end{thm}
\noindent{\it Sketch of the Proof.} By linearity, it suffices to show that, 
$$u \in {\mathcal P}_{\Lambda}^{1}(r) \cap {\mathcal W}_{M}^{1}(\bar C r),$$ 
for some appropriate $\Lambda$ and $M$ large independent of $\eps$ (to be specified below) and with $r \leq c_0$ and $\bar C$ universal as given in Theorem \ref{main}.

Here we work with the graph of $u$ (see Remark \ref{r3}) and the inequality in the class ${\mathcal W}_{M}^{a}(\rho)$ is understood to hold everywhere (with $\delta=0$) instead of in the $\mathcal H^n$-measure sense. Another alternative way is to think that $u$ is extended from the lattice to the whole space by some linear interpolation.

It is straightforward to see that for  $\Lambda$ large
$$u \in {\mathcal P}_{\Lambda}^{1}(2\eps).$$
Indeed, if $$P_\Lambda(x)=\frac \Lambda 2 (x \cdot \xi)^2- \frac 12 |x|^2+ l(x), \quad \quad \mbox{$l$ linear},$$ 
touches $u$ from below, say at $0$ in $B_{2\eps},$
and we choose $\Lambda$ such that $$(\Lambda/n-1)\lambda_{min}>(n-1)\lambda_{max},$$ we obtain that $Lu(0) \ge LP_\Lambda(0)>0$ and get a contradiction.

Also, if 
$$u\geq 0 \quad \text{in $B_\rho\cap \eps \mathbb Z^n$}, \quad \rho \leq K \eps$$ and $$u(0)\leq 1,$$ it immediately follows from the fact that $Lu(0)=0$ that $$u(\eps e_i)\leq C(n, \lambda_{min}, \lambda_{max}).$$
Hence,
$$u \leq C(K) \quad \text{in $B_\rho$}.$$ Choosing $K=2\bar C$, $M=C(K)$ and $r= 2 \eps$ we obtain that $u \in {\mathcal W}^1_M(\bar C r)$ as desired.
\qed

\subsection{Homogenization.} Next we obtain H\"older estimates in the context of homogenization by considering linear equations with measurable coefficients whose ellipticity constants might degenerate in the interior of $Q_\eps$ the cubes of size $\eps$.  
Let $L_\eps u$ be a second order operator 
\begin{equation}\label{hom}
L_\eps u(x):= a^{ij}(\frac{x}{\eps}) u_{ij}(x)
\end{equation}
with $A(x):=a^{ij}(x)$ continuous coefficients in $\R^n$ satisfying 
\begin{equation}\label{hom_elliptic}
\tilde\lambda_{min}(x) I \leq A(x) \leq \tilde\lambda_{max}(x) I, \quad A(x) \neq 0,
\end{equation}
and 
$$\tilde \lambda_{min}, \tilde \lambda_{max} \geq 0 \quad \text{periodic of period $1$}$$
with
\begin{equation}\label{away}0<\lambda_{min} \leq \tilde \lambda_{min}(x) \leq \tilde \lambda_{max}(x)\leq \lambda_{max}, \quad \text{in $Q_1 \setminus D$},\end{equation}
for some set $D \subset \subset Q_1.$

Similarly as in the previous application, we can obtain the following result.

\begin{thm}\label{THo} Assume that 
$$L_\eps u=0, \quad |u|\leq 1 \quad \text{in $B_1$}.$$ Then,
$$|u(x)-u(y)| \leq C |x-y|^\gamma, \quad \text{in $B_{1/2}$} \quad \mbox{if $|x-y| \ge \eps$,}$$
with $C, \gamma$ depending on $n$, $\lambda_{min}$, $\lambda_{max}$ and $dist(\bar D, \p Q_1)$.
\end{thm}
\noindent{\it Sketch of the Proof.} It suffices to show that
$$u \in {\mathcal P}_{\Lambda}^{1}(r) \cap {\mathcal W}_{M}^{1}(\bar C r),$$ 
as long as $\eps \le c$ for some appropriate $\Lambda$, $M$, $r$.
We rescale $u$ as
$$\tilde u(x)= \eps^{-2} u(\eps x).$$
Then, $$L \tilde u:= a^{ij}(x)\tilde u_{ij}(x)=0 \quad \text{in $\R^n.$}$$ By assumption \eqref{away}, a polynomial of the form $ P_\Lambda$ cannot touch $\tilde u$ in $Q_1 \setminus D$, as long as $\Lambda$ is large enough (depending only on $n, \lambda_{min}, \lambda_{max}$) so that $L \tilde P_\Lambda >0$. Now, assume for simplicity that $P_\Lambda=\frac \Lambda 2 x_n^2 - \frac 12 |x|^2$ touches $\tilde u$ from below say at $0 \in D$ in $B_{2\sqrt n}$. Then, $$\tilde u \geq - 2 n \quad \text{in $B_{2\sqrt n}$,}$$
and for $c>0$ depending on $dist(\overline{D}, \p Q_1)$
$$\tilde u(\bar x) + 2 n\geq c\Lambda , \quad \text{at some $\bar x \in \p Q_{1}$}.$$ By Harnack inequality applied to $\tilde u+ 2n$ in a neighborhood of $\p Q_1$ (depending on $dist(\overline D, \p Q_1)$) we get that,
$$\tilde u + 2 n \geq \tilde c \Lambda \quad \text{on $\p Q_{1}$.}$$ Choosing $\Lambda$ large this gives
$$\tilde u \geq c' >0 \quad \text{on $\p Q_1$}.$$ By the maximum principle the inequality holds also in $Q_1$ hence $\tilde u (0)>0$, a contradiction. Rescaling back this shows that
$$u \in {\mathcal P}_\Lambda^1(r), \quad \quad r:=2\sqrt n \eps.$$ 

To prove the other inclusion, assume
$$\tilde u (x_0)=1,$$ for some $x_0 \in Q_1$. Again, by the maximum principle,
$\tilde u (\bar x) \leq 1$ at some $\bar x \in \p Q_1$, and by Harnack inequality,
$\tilde u \leq C(K)$ on $\p Q_K$ and therefore in $Q_K$ as well. We take $K=8 n \bar C$ and we obtain
$\tilde u \leq M$ in $B_{2 \sqrt n \bar C }(x_0)$. Thus,
$$u \in {\mathcal W}^1_M(\bar Cr),$$
as desired.

\qed

\begin{rem}
Assume that $u$ satisfies an oscillatory fully nonlinear equation
$$F(D^2u, \frac x \eps)=0,$$
with $F$ an elliptic operator which is periodic in the second variable, and with ellipticity constants which might degenerate when the second variable belongs to a set $D$ as above. Then Theorem \ref{THo} can be applied to differences between translations of $u$, and then it follows that $u$ satisfy a uniform $C^{1,\alpha}$ interior estimate up to scale $\eps$.  
\end{rem}

\subsection{Non-local operators} We show that solutions to integro-differential equations with measurable kernels satisfy the hypotheses of Theorem \ref{main}.

Let $\sigma \in (0,2)$ and let 
$K_x(y)$ be a symmetric, measurable kernel proportional to the kernel $K_\sigma$ of $\triangle ^{\sigma/2}$
$$\lambda_{min} \, \, K_\sigma(y) \le K_x(y) \le \lambda_{max} \, \, K_\sigma(y), \quad \quad K_x(y)=K_x(-y),$$
with $\lambda_{min}$, $\lambda_{max} >0$ and
$$K_\sigma(y):=(2-\sigma)|y|^{-n-\sigma}.$$

Consider the integro-differential operator
$$\mathcal Lu(x):= \int_{\R^n}(u(x+y)-u(x))K_x(y)dy$$ and let
$u$ be a viscosity solution to 
$$\mathcal Lu = f.$$
For definitions and properties of viscosity solutions to integro-differential equations as above, we refer the reader to \cite{CS}.

As an application of our Theorem \ref{main}, we recover the H\"older continuity property of $u$, with uniform estimates as $\sigma \to 2$, which is due to Caffarelli and Silvestre \cite{CS}.

\begin{thm}\label{nonlocal} Let $u$ be a viscosity solution to 
\begin{equation} \mathcal L u = f \quad \mbox{in $B_1$}. \end{equation} 
Then, \begin{equation}\|u\|_{C^\alpha(B_{1/2})} \leq C \left (\|f\|_{L^\infty} + \|u\|_{L^\infty(B_1)} + (2-\sigma)\int_{\R^n \setminus B_1} \frac{|u(y)|}{|y|^{n+\sigma}} dy \right). \end{equation} The constants $\alpha>0,$ and $C$ depend only on $n$, $\lambda_{min}$, $\lambda_{max}$ if $\sigma$ is close to 2.
 \end{thm} 
 
 Once we establish the following lemma, the proof follows easily by scaling and it is left to the reader. 
We assume that $\sigma$ is sufficiently close to $2$.
 
 \begin{lem}\label{HI_nonlocal}  Let $u$ be a continuous function satisfying
 $$\mathcal Lu(x) \leq 1, \quad \mbox{and} \quad u \geq 0 \quad \text{in $B_1$}$$
 and
 \begin{equation}\label{infinity}
\int_{\R^n \setminus B_1} u^- \, \,  K_\sigma dy \leq 1.\end{equation}
Then there exists $\Lambda$ depending on $n, \lambda_{min}, \lambda_{max}$, such that 
$$u \in {\mathcal P}_{\Lambda}^{a}(r) \cap {\mathcal W}_{2}^{a}(\bar C r), \quad \quad \forall a \in [1,\infty),$$
for some $r$ small, and with $\bar C=\bar C(n, \Lambda)$ the constant from Theorem $\ref{HI}.$
\end{lem}

\begin{proof} We remark that 
\begin{equation}\label{41}
 \int_{B_1} |y|^2 \, K_\sigma dy =(2-\sigma)\int_0^1 t^{2-n-\sigma} \omega_nt^{n-1}dt =\omega_n,
\end{equation}
where $\omega_n$ represents the surface area of the unit sphere in $\R^n$.

Let us choose $r$ so that
\begin{equation}\label{42}
\int_{B_{r}} |y|^2 \, K_\sigma dy= \frac 1 2 \int_{B_1} |y|^2 \, K_\sigma dy = \frac {\omega_n}{2}\quad \Longrightarrow \quad r^{2-\sigma}= \frac 12,
\end{equation}
We check that $$u \in \mathcal P^a_\Lambda(r), \quad \quad \forall a \in [1, \infty),$$
for some large $\Lambda$. It suffices to consider the case $a=1$, since the general case follows after dividing by $u$ by $a$. Assume by contradiction that (see Definition \ref{p2}) 
$$ P(x) + \frac \Lambda 2 ((x-x_0) \cdot \xi)^2 \chi_{B_r (x_0)},$$
touches $u$ from below at $x_0$ in $B_1$, where $P=P_{y_0}^1$ is a quadratic polynomial of opening $-1$ that is below $0$ outside $B_{3/4}$. We reach a contradiction since
\begin{align*}
\mathcal Lu(x_0)  \ge  \int _{B_r(x_0)} & (u-P)(y)K_{x_0}(y) dy + \int_{B_{1/4} (x_0)}  (P(y)- P(x_0))K_{x_0}(y)dy  \\
& + \int_{\mathbb{R}^n \setminus B_{1/4}(x_0) }(u(y)-u(x_0))K_{x_0}(y)dy,\\
 \ge c(n) \Lambda & \lambda_{min} - C(n) \lambda_{max}- C(n) \lambda_{max} >1 
\end{align*}
where we have used that $u(x_0)=P(x_0) \le 1$, and \eqref{infinity}, \eqref{41}, \eqref{42} together with
$$ \int _{\mathbb{R}^n \setminus B_{1/4}} K_\sigma dy \le C(n).$$

Next we check that 
\begin{equation}\label{43}
u \in \mathcal W^a_2(\rho), \quad \quad \forall a \in [1, \infty), \quad \mbox{if} \quad (2-\sigma) \rho^{-\sigma} \ge C_0.
\end{equation}
As above, it suffices to consider the case $a=1$ and assume that a polynomial $P=P^1_{y_0}$ touches $u$ by below at $x_0$. Using the computation above, together with $\mathcal Lu(x_0) \le 1$ (and replacing $r$ by $\rho$) we find 
 $$C \ge (2-\sigma)\int_{B_{\rho}(x_0)}\frac{(u-P)(y)}{|y|^{n+\sigma}} dy,$$
for come $C$ depending on $n$, $\lambda_{min}$, $\lambda_{max}$.
If $$\mu:=\frac{|\{u-P > 2\} \cap B_{\rho}(x_0)|}{|B_\rho|},$$ then we obtain
$$C \ge 2 (2-\sigma) \mu |B_1| \rho^{-\sigma},$$
and we obtain the desired conclusion $\mu \le \delta(n)$ provided that we choose $C_0$ sufficiently large.

Finally, it is straightforward to check that $\rho=\bar C r$ with $r$ defined in \eqref{42} satisfies the inequality in \eqref{43} if $\sigma$ is close to 2.

\end{proof}

\begin{rem}
If $ \sigma$ is not close to 2 (or if we allow the constants to depend on $\sigma$) then the weak Harnack inequality is satisfied at scale 1, and there is no need to consider the class $\mathcal P$. Indeed, the last argument gives $u \in \mathcal W_M^a(1/2)$ for an appropriate constant $M$ and all $a \ge 1$, and the conclusion of Theorem \ref{HI} is already satisfied. 
\end{rem}

\section{Quasi-Minimizers of perimeter} 

Let $E$ be a measurable set and $\Omega$ an open set in $\R^n$. We denote by $P(E, \Omega)$ the so-called perimeter of $E$ in $\Omega$ and refer to \cite{G} for definitions and properties of the perimeter functional and Caccioppoli sets.

We say that $E$ is a quasi-minimizer of the perimeter functional in $\Omega$ if there exist constants $\kappa,\alpha>0$ such that
\begin{equation}P(E, B_\rho(x)) \leq (1+\kappa\rho^{\alpha})P(F,B_\rho(x))
\end{equation}
for all $x \in \p E$, and all measurable sets $F$ which coincide with $E$ outside $B_{\rho}(x)\subset \Omega.$

Classical results of \cite{A,T,Bo} give that the boundary of a quasi-minimizer $E$ can be split into the union of a $C^{1,\alpha/2}$
relatively open hypersurface and a closed singular set of Hausdorff dimension at most $n-8$. These results extend the well-know theory for minimizers, due to De Giorgi.

%For the proof of a classical compactness result for quasi-minimizers and several other standard properties we refer the reader to \cite{A}.

Here we want to use the Harnack inequality established in Section 2 to show that quasi-minimizers enjoy a flatness regularity theory, as in  the classical case of minimizers, using a non-variational approach. Precisely, the following theorem holds and can be obtained with a slight modification of our arguments from Section 2, combined with the technique developed in \cite{S} for perimeter minimizers.

\begin{thm}\label{quasi} Assume that the open set $E$ with  
\begin{equation}\label{flat}\{x_n \leq -\eps\} \cap \mathcal C \subset E \subset \{x_n \leq -\eps\}\end{equation}
is a quasi-minimizer in the cylinder $\mathcal C:= \{|x'| < 1\} \times \{|x_n|<1\}.$ Then $$\p E \cap \{|x'|<1/2\}$$ is a $C^{1,\gamma}$ surface if $\eps$ is small enough depending only on $n, \kappa, \alpha.$ 
\end{thm}
The proof of Theorem \ref{quasi} is based on an improvement of flatness lemma for $\p E$ as in \cite{S}.
Without loss of generality, after a dilation, we can assume $\kappa=1.$ Constants depending only on $n, \alpha$ will be called universal. 

We remark that if $E$ is a quasi-minimizer and $0 \in \p E$ then  
$$P(E,B_r) \le 2 P(F,B_r), \quad \quad \quad \forall r \le 1,$$
and by taking $F=E \setminus B_r$ it can be shown by standard arguments that (see \cite{G}) 
\begin{equation}\label{50}
P(E,B_r) \le C(n) r^{n-1}, \quad \quad |E \cap B_r| \ge c(n) r^n.
\end{equation}

  From now on $E$ will be a quasi-minimizer to the perimeter functional in the unit cylinder $\mathcal C$ with $\kappa=1$. 
Also, balls in $\R^{n-1}$ will be denoted with $B'$. Finally, given $r>0$ (small) we denote by $E_r$ the $1/r$ dilation of $E$
 $$E_r := \frac 1 r E.$$
 
 The desired improvement of flatness lemma is stated below.
 
 \begin{lem}\label{imp} There exist $r_0, \eps_0>0$ small universal such that if 
\begin{equation}\label{inclusion1}
\p E_r \cap B_{1} \subset \{|x_n| \leq r^{\eps_0}\} , \quad 0 \in \p  E_r, \quad r \leq r_0
\end{equation}
then in some new system of coordinates $(y', y_{n})$ 
\begin{equation}\label{inclusion2}
\p E_r \cap B_{\eta} \subset \{|y_n| \leq \eta^{1+\eps_0}r^{\eps_0}\} 
\end{equation}
for some $\eta>0$ small universal.
 \end{lem}

The proof of this lemma relies on a compactness argument, once the appropriate Harnack-type inequality is obtained. 
 The Harnack inequality reads as follows.

\begin{lem}\label{HI1} There exist $\eps_0, c_0>0$, universal such that if for $r>0$ small
\begin{equation}\label{50.1}
\p E_r \cap (B'_1 \times [-1,1]) \subset \{|x_n| \leq a\}, \quad a \in [r^{\eps_0}, c_0]
\end{equation}
then 
\begin{equation}\label{50.2}
\p E_r \cap (B'_{1/2} \times [-1,1]) \subset \{ a' \le x_n \le a''\},
\end{equation}
with either $a'=(\eta_0-1)a$ or $a''=(1-\eta_0)a$, for some $\eta_0>0$ small universal.
\end{lem}

Before we proceed with the proof of Lemma \ref{HI1}, we show first that $\p E$ belongs to an appropriate $\mathcal P_\Lambda^I$ class.

\begin{lem}\label{classP} Let $E$ be a quasi-minimizer and $\rho>0$ small. Then $\p E$ cannot be touched by below at $0$ in the cylinder $B'_\rho \times [-\rho, \rho]$ with the graph
$$x_n= \gamma \, \, Q(x'), \quad \quad \quad Q(x'):=2n x_1^2- \frac 12 |x'|^2 + b' \cdot x',$$
if $$\gamma \in [\rho^{-1+\mu},  c_1\rho^{-1}], \quad \quad \gamma |b'| \le c_1,$$ 
and $c_1, \mu>0$ universal.
\end{lem}

\begin{proof} 
Assume that $E$ lies above the graph of $ \gamma \, Q$. Consider the graph
$$\Gamma:=\{ \quad x_n= \psi(x')\}, \quad \quad \psi(x'):=\gamma \left(Q + \frac 14 \rho^2- \frac 12 |x'|^2\right),$$
with $\gamma \le c_1 \rho^{-1}$, and define the competitor $F$ 
 $$F:=E \cap \{x_n >\psi(x'\}.$$

 We claim that
 \begin{equation}\label{2}
 P(E, B_\rho) - P(F, B_\rho) \geq c(n) \gamma \,  |A|,
 \end{equation} for a universal constant $c(n)$, with
  $$A := E \cap \{x_n < \psi(x')\} \subset \subset B_{\rho}.$$

\begin{figure}[h]
\includegraphics[width=0.8 \textwidth]{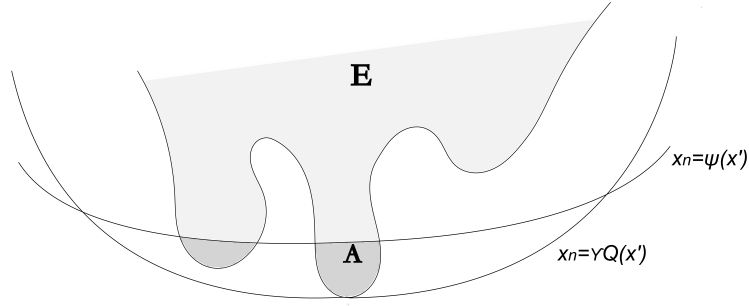}
\caption{The region $A$}
    \label{fig2}
\end{figure}

To derive \eqref{2} let us denote by $d_\Gamma$ the signed distance function from $\Gamma$, which is positive in $F$ and compute
 $$\int_A \Delta d_\Gamma \; dx = \int_{\p A} (d_\Gamma)_\nu \;d\mathcal H^{n-1} \geq \mathcal H^{n-1}(\Gamma\cap E) - P(E, \{x_n < \psi\}) =$$ $$= P(F, B_{\rho})- P(E, B_{ \rho}).$$
 Since,
 $$\Delta d_\Gamma(x)= H_{\Gamma_p}(x), \quad x\in A$$
 where $\Gamma_p$ is the parallel surface to $\Gamma$ passing through $x$, it suffices to show that
 $$ H_{\Gamma_p} \geq c(n) \, \gamma \quad \text{in $A$}.$$ 
This follows from standard computations (see \cite{GT}) by using that 

a) $|\nabla \psi|\le C(n) \gamma \rho + c_1 \le c_1 C(n)$, 

b) the principal curvatures of $\Gamma$ are bounded by $C \gamma$, 

c) $|d_\Gamma| \le C \rho^2 \ll (C \gamma)^{-1}$ in $A$ 

and
$$H_\Gamma(x) = div\frac{\nabla \psi}{ \sqrt{1+|\nabla \psi|^2}} \geq 2 \Delta \psi - C(n) (\max |D^2\psi|) |\nabla \psi|^2  \geq c \, \, \gamma,$$
provided that $c_1$ is chosen sufficiently small. Thus the desired bound \eqref{2} is proven.

 Now, since $E$ is a quasi-minimizer we get that 
 $$P(E, B_\rho)- P(F,B_{\rho}) \leq \rho^{\alpha}P(F,B_{\rho}) \leq  \rho^{\alpha}P(E, B_\rho) \leq C \rho^{\alpha+n-1}$$
 where in the last inequality we used \eqref{50} with $r= \rho$.
 Combining this estimate with \eqref{2} we get that
 \begin{equation}\label{3}
 c \,  \gamma \, |A| \leq C \rho^{\alpha+n-1}.
 \end{equation}
 By the second density estimate in \eqref{50} we have
\begin{equation}\label{density}|A| \geq |E \cap B_{ \gamma \rho^2/8}| \geq c_0 (\gamma \rho^2)^n.\end{equation}
Here we have used that by the Lipschitz continuity of $\Gamma$
$$B_{\gamma \rho^2/8} \subset \{x_n <\psi(x')\} \cap B_\rho.$$
Then \eqref{density} together with \eqref{3} gives 
 $$\gamma \leq C \rho^{-1 + \frac{\alpha}{n-1}} \le \rho^{-1+\mu}.$$

\end{proof}

\begin{rem}\label{r5} We remark that the proof above holds if we replace the quadratic part of $Q$ by $\frac 12 (x')^T M x'$ with $tr M>0$ provided that the constant $c_1(M)$ is chosen sufficiently small depending on $M$ as well. In particular the estimate \eqref{3} is valid with a constant $c=c(M)$ as long as $\gamma \rho \le c_1(M)$.
\end{rem}

We have the following corollary of Lemma \ref{classP}.

\begin{cor}\label{pE}
Assume that $\p E$ satisfies the hypothesis \eqref{50.1} of Lemma $\ref{HI1}$. Then in $B_1'$ we have
$$\p E_r + a e_n  \, \in \, \mathcal P^I_{4n}(a), \quad \quad I=[a^2, c_1 ].$$
\end{cor}

\begin{proof}
If $x_n= \sigma P_{4n}(x')$ with $\sigma \in I$ touches from below $\p E_r$ at some point, say on $x'=0$, in the cylinder $|x'| \le a$, then $$x_n= \sigma \,  r \, P_{4n}(x' r^{-1}) =: \sigma \,  r^{-1} Q(x'),$$ touches $\p E$ from below in the cylinder $|x'| \le ar=:\rho$, with $Q$ a polynomial as in Lemma \ref{classP}. We contradict the conclusion of the lemma since it is straightforward to verify that
$$\sigma r^{-1} \in [\rho^{\mu-1},c_1 \rho^{-1}] \quad \mbox{and} \quad \sigma \le c_1,$$
provided that $\eps_0$ and $c_0$ are sufficiently small. 

\end{proof}

{\it Proof of Lemma $\ref{HI1}$}.
If $\p E_r$ does not lie above $x_n=(\eta_0-1)a$ in $B_{1/2}'$ then, by Corollary \ref{pE}, we may apply Proposition \ref{PP} to $\p E_r + a e_n$.
 
If $[\eta_0 a, C \eta_0 a] \subset [a^2 , c_1]$ with $C$ universal, then
$$\mathcal H^{n-1} \left(A^{l}_{C \eta_0 a}(\p E_r) \cap B_{1/2}'\right) \ge \frac 3 4 \, \, \mathcal H^{n-1}(B_{1/2}'),$$
where $$A^{l}_{C \eta_0a}(\p E_r)$$ represents the collections of cubes $Q_l'$ of size $\bar C a$ (with $\bar C$ universal) in the $x'$ variable, in which the contact set $A_{C \eta_0a }(\p E_r)$ projects.

We choose $\eta_0$ such that $C \eta_0 = 1/2$, and then $a \le c_0$ small such that $a^2 \le \eta_0 a.$ Denote by 
$A^-_l(\p E_r)$ the union of the cubes $Q'_l$ of size $\bar C a$ (small) with the property that 
\begin{equation} \label{53}
Q'_l \times [-a, -\frac 1 2 a] \, \,   \cap \p E_r \neq \emptyset.
\end{equation}
Since $A^{l}_{C \eta_0a}(\p E_r) \subset A^-_l(\p E_r)$, the inequality above gives
\begin{equation}\label{variant}
\mathcal H^{n-1} \left(A^-_l(\p E_r) \cap B'_{1/2} \right) \geq \frac 3 4 \mathcal H^{n-1}(B'_{1/2}). 
\end{equation}

If $\p E_r$ does not lie below $x_n=(1-\eta_0)a$ in $B_{1/2}'$ then, after using the same argument ``upside-down", we find that 
\begin{equation}\label{upsidedown}
\mathcal H^{n-1} \left(A^+_l(\p E_r) \cap A^-_l(\p E_r) \cap B'_{1/2} \right) \geq \frac 12 \mathcal H^{n-1}(B'_{1/2}),
\end{equation}
where $A^+_l(\p E_r)$ denotes the union of cubes $Q'_l$ with the property that 
\begin{equation} \label{54}
Q'_l \times [\frac 1 2 a,  a] \, \, \cap \p E_r \neq \emptyset.
\end{equation}
On the other hand, by projecting $\p E_r$ along $e_n$, we find that for all cubes $Q'_l$,
\begin{equation}\label{55}
P(E_r, Q'_l \times [-a, a]) \geq  \mathcal H^{n-1}(Q'_l),
\end{equation}
with equality only if $\p E_r$ coincides with a hyperplane $x_n = const.$ in the cylinder $Q_l' \times [-a,a]$.  
This inequality can be improved as
\begin{equation} \label{56}
P(E_r, Q'_l \times [-a, a]) \geq (1+ c_2)\mathcal H^{n-1}(Q'_l), \quad \mbox{if} \quad Q'_l \in A^+_l(\p E) \cap A^-_l(\p E).\end{equation}
This follows by compactness and the fact that \eqref{53}, \eqref{54} 
together with the density estimate \eqref{50} imply that 
both $E_r$ and its complement $E^c_r$ have uniform density estimates 
in the two half cylinders $Q_l' \times [-a,0]$ and $Q_l' \times [0,a]$, hence $E_r$ cannot be close (say in $L^1$ sense) to a half space in the $e_n$ direction.

Therefore, using \eqref{55} combined with \eqref{56},\eqref{upsidedown} we obtain 
\begin{equation}\label{sumup}
P(E, B'_{1/2} \times  [-a, a]) \geq \left(1+ \frac {c_2}{4} \right) \mathcal H^{n-1}(B'_{1/2}).
\end{equation}
Now we compare $E$ and $F:=E \setminus (B'_{r/2} \times [-ra, \infty))$ in $B_r$, and by quasi-minimality,
$$P(E,  B'_{r/2} \times [-\frac r 2 \times \frac r 2]) \le (1+ r^{\alpha}) (1+ Ca)\mathcal H^{n-1}(B'_{r/2}) + C r^{\alpha + n-1}.$$
This gives
$$P(E_r, B'_{1/2} \times  [-a,a]) \leq (1+ C a + C r^{\alpha})\mathcal H^{n-1}(B'_{1/2}),$$
and we contradict \eqref{sumup} if $a$ is small enough.
\qed

\bigskip

Next we provide the proof of the improvement of flatness Lemma \ref{imp}.

\bigskip

{\it Proof of Lemma $\ref{imp}.$} We argue by compactness. Assume by contradiction that there exist a sequence $r_k \to 0$ and of quasi-minimizers $E^k$ such that $E^k_{r_k}$ satisfies the assumption of the lemma but not the conclusion.

Denote by $$E^k_*:= \left\{(x', r_k^{-\eps_0}x_n) : x \in \p E^k_{r_k} \right \} \, \, \subset B_1' \times [-1,1].$$ 
Then, in view of Lemma \ref{HI1} (up to extracting a subsequence), $E^k_*$ converges in the Hausdorff distance sense to the graph $E^*:=\{x_n=w(x')\}$ of a H\"older continuous function $w$. We wish to prove that $w$ is harmonic. This leads to a contradiction since the graph of a harmonic function satisfies the improvement of flatness with respect to linear functions.

To show that $w$ is harmonic, let us assume by contradiction that a quadratic polynomial $Q(x')$ with $$\Delta Q > \delta >0$$ 
touches $w$ strictly from below, say for simplicity at $0$, in a neighborhood $B'_\delta$. Then, by the uniform convergence of $E^k_*$ to $E_*$ we conclude that for $k$ large (rescaling back)
\begin{equation}\label{inclusionk}E^k \subset \{x_n > Q_k:= r_k^{1+\eps_0}(Q(r^{-1}x')-c_k)\}, \quad c_k \to 0 \quad \text{as $k \to \infty$}.\end{equation}

Notice that the polynomial $Q_k$ can be written as 
$$Q_k= \gamma_k (\tilde Q +o(1)), \quad \quad \gamma_k:= r_k^{\eps_0 -1},$$
and $\tilde Q$ satisfies the hypotheses of Remark \ref{r5} with $\rho_k:= \delta r_k$. We obtain  
\begin{equation}\label{density3}c(\delta) \gamma_k |A_k| \leq C \rho_k^{\alpha+n-1},\end{equation}
where $A_k$ is defined as in the proof of Lemma \ref{classP}.
On the other hand, in view of the H\"older continuity of $w$ and the uniform convergence of the $E^k_*$ to $E_*$, we have 
$$|A_k| \geq c(\delta) \rho_k^{n-1}(\gamma_k \rho_k^2).$$ Therefore, we reach a contradiction if 
$$c(\delta) \gamma_k^2 \rho_k^{n+1} \geq C \rho_k^{\alpha+n-1}.$$
This is the case for $k$ large, as long as 
$$\eps_0 < \frac \alpha 2.$$
\qed


\begin{thebibliography}{9999}
 \bibitem[A]{A} Almgren F. J. Jr., Existence and regularity almost everywhere of solutions to elliptic variational problems with constraints. Mem. Amer. Math. Soc. 4 (1976), no. 165, viii+199 pp. 
\bibitem[B]{B}Bakelman I.J., Convex Analysis and Nonlinear Geometric Elliptic Equations, Springer-Verlag, New York, 1994.
\bibitem[BLP]{BLP}Bensoussan, Alain; Lions, Jacques-Louis; Papanicolaou, George
Asymptotic analysis for periodic structures.
Studies in Mathematics and its Applications, 5. North-Holland Publishing Co., Amsterdam-New York, 1978. xxiv+700 pp. ISBN: 0-444-85172-0
\bibitem[Bo]{Bo}  Bombieri E.,  Regularity theory for almost minimal currents. Arch. Rational Mech. Anal. 78 (1982), no. 2, 99Ð130. 
\bibitem[CC]{CC} Caffarelli L., Cabr\'e X., 
Fully nonlinear elliptic equations. (English summary)
American Mathematical Society Colloquium Publications, 43. American Mathematical Society, Providence, RI, 1995. vi+104 pp. ISBN: 0-8218-0437-5

 \bibitem[CS]{CS} Caffarelli L.,  Silvestre L.,  Regularity theory for fully nonlinear integro-differential equations. Comm. Pure Appl. Math. 62 (2009), no. 5, 597--638.
 
\bibitem[GT]{GT} Gilbarg D., Trudinger N.S., Elliptic Partial Differential Equations of Second Order, 2nd edition, Springer-Verlag, New York, 1983. 
 \bibitem[G]{G}Giusti E.,  Minimal surfaces and functions of bounded variation. Monographs in Mathematics, 80. BirkhŠuser Verlag, Basel, 1984. xii+240 pp. ISBN: 0-8176-3153-4 
\bibitem[HT]{HT}Hung-Ju K., Trudinger N.S.,  Maximum principles for difference operators. Partial differential equations and applications, 209--219, Lecture Notes in Pure and Appl. Math., 177, Dekker, New York, 1996.
\bibitem[S]{S}  Savin O.,  Small perturbation solutions for elliptic equations. Comm. Partial Differential Equations 32 (2007), no. 4-6, 557--578. 
\bibitem[T]{T} Tamanini I.,  Boundaries of Caccioppoli sets with H\"older-continuous normal vector. J. Reine Angew. Math. 334 (1982), 27--39. 
\end{thebibliography}
\end{document}